\documentclass[12pt, a4paper]{amsart}
\usepackage{hyperref, fullpage, amsmath, amsfonts, amsthm, amssymb, stmaryrd}

\theoremstyle{plain}
\newtheorem{Theorem}{Theorem}
\newtheorem{Lemma}[Theorem]{Lemma}

\newcommand{\auau}{{$\hhb{-2}_{\hbox{\large\~{}}}\hhb{-1}$}}
\newcommand{\hhb}[1]{{\hbox to#1pt{}}}
\newcommand{\euS}{{\mathfrak{S}}}
\newcommand{\KS}{K^\euS}
\font\tensc=cmcsc12 scaled\magstep0

\address{School of Mathematical Sciences, Tel Aviv University, 
Ramat Aviv, Tel Aviv 69978, Israel}
\email{barylior@post.tau.ac.il}

\address{Fachbereich Mathematik und Statistik, University of 
Konstanz, 78457 Konstanz, Germany}
\email{arno.fehm@uni-konstanz.de}

\address{Department of Mathematics, University of Pennsylvania, 
DRL, 209 S 33rd Street, Phila\-delphia, PA 19104, USA}
\email{pop@math.upenn.edu \ \ \ {\it URL}: http://math.penn.edu/\auau pop}

\begin{document}

\title{\large On fields of totally $\euS$-adic numbers
\vskip7pt
{\small\bf --- With an appendix by Florian Pop ---}}
\author{Lior Bary-Soroker \and Arno Fehm}
\thanks{The authors are indebted to Pierre D\`ebes for pointing out 
to them the result in \cite{DebesHaran}. They would also like to thank 
Sebastian Petersen for motivation to return to the subject of this note.
This research was supported by the Lion Foundation Konstanz -- 
Tel Aviv and the Alexander von Humboldt Foundation.}

\begin{abstract}
Given a finite set $\euS$ of places of a number field,
we prove that the field of totally $\euS$-adic algebraic numbers 
is not Hilbertian.
\end{abstract}

\maketitle

\noindent
The field of totally real algebraic numbers $\mathbb{Q}^{\rm tr}$, the field of totally 
$p$-adic algebraic numbers $\mathbb{Q}^{{\rm t}p}$, and, more generally, fields 
of totally $\euS$-adic algebraic numbers $\mathbb{Q}^\euS$, where $\euS$ is 
a finite set of places of $\mathbb{Q}$, play an important role in number theory 
and Galois theory, see for example \cite{FHV94, Pop, Taylor, HJPe}.
The objective of this note is to show that none of these fields is Hilbertian
(see \cite[Chapter 12]{FriedJarden} for the definition of a Hilbertian field).

Although it is immediate that $\mathbb{Q}^{\rm tr}$ is not Hilbertian,
it is less clear whether the same holds for $\mathbb{Q}^{{\rm t}p}$.
For example, every finite group that occurs as a Galois group over 
$\mathbb{Q}^{\rm tr}$ is generated by involutions (in fact, the converse 
also holds, see \cite{FHV93}) although over a Hilbertian field all finite 
abelian groups (for example) occur. In contrast, over $\mathbb{Q}^{{\rm t}p}$ 
every finite group occurs, see \cite{Efrat}. In fact, although (except in the 
case of $\mathbb{Q}^{\rm tr}$) it was not clear whether these fields are 
actually Hilbertian, certain weak forms of Hilbertianity were proven and 
used, both explicitly and implicitly, for example in \cite{FHV93, HJPd}.
Also, any proper finite extension of any of these fields is actually 
Hilbertian, see \cite[Theorem 13.9.1]{FriedJarden}.

The non-Hilbertianity of $\mathbb{Q}^{{\rm t}p}$ was implicitly 
stated and proven in \cite[Examples 5.2]{DebesHaran} but this result 
seems to have escaped the notice of the community and was forgotten. 
We give a short elementary proof (which is closely related to the proof in 
\cite{DebesHaran}) of the following more general result.

\begin{Theorem}\label{thm}
For any finite set $\euS$ of real archimedean or ultrametric discrete 
absolute values on a field $K$, the maximal extension $\KS$ of $K$ in 
which every element of $\euS$ totally splits is not Hilbertian.
\end{Theorem}

Note that $\KS$ is the intersection of all Henselizations and real 
closures of $K$ with respect to elements of $\euS$. We would like to 
stress that $\euS$ does not necessarily consist of {\em local} primes 
in the sense of \cite{HJPe}.

After this note was written, it turned out that there is an unpublished 
manuscript of Pop with a different proof of Theorem \ref{thm} (which 
is less explicit but works in a more general setting), see the~Appendix 
at the end of this paper.

Let 
\begin{eqnarray*}
 \gamma(Y,T)&=&(Y^{-1}+T^{-1}Y)^{-1} \,\;=\;\,\frac{YT}{Y^2+T}
\end{eqnarray*}
and
\begin{eqnarray*}
 f(X,Z) &=& X^2+X-Z^2. \\
\end{eqnarray*}

\begin{Lemma}\label{lem:vg0}
If $(F,v)$ is a discrete valued field with uniformizer $t\in F$, then
$v(\gamma(y,t))>0$ for each $y\in F$.
\end{Lemma}

\begin{proof}
If $v(y)=0$, then $v(t^{-1}y)<0=v(y^{-1})$, so $v(y^{-1}+t^{-1}y)<0$.
If $v(y)<0$, then $v(y^{-1})>0$ and $v(t^{-1}y)<0$, so $v(y^{-1}+t^{-1}y)<0$.
If $v(y)>0$, then $v(y^{-1})<0$ and $v(t^{-1}y)\geq 0$ since $t$ is 
a uniformizer, so again $v(y^{-1}+t^{-1}y)<0$. Thus, in each case, 
$v(\gamma(y,t))=-v(y^{-1}+t^{-1}y)>0$.
\end{proof}

\begin{Lemma}\label{lem:irred}
Let $F$ be a field and $t\in F\smallsetminus\{0,-1\}$. 
If ${\rm char}(F)=2$, assume in addition that $t$ is not a square in $F$.
Then $f(X,\gamma(Y,t))$ is irreducible over $F(Y)$.
\end{Lemma}

\begin{proof}
If ${\rm char}(F)\neq 2$, then
$f(X,\gamma(Y,t))$ is reducible if and only if the discriminant 
$1+4\gamma(Y,t)^2$ is a square in $F(Y)$. This is the case if and 
only if $(Y^2+t)^2+4(tY)^2$ is a square. 
Writing 
\[
 (Y^2+t)^2+4(tY)^2  = (Y^2+aY+b)^2
\] 
and comparing coefficients we get that
$a=0$,  $b^2=t^2$, and $a^2+2b=2t(1+2t)$. Hence, $t=0$ or $t=-1$.
 
If ${\rm char}(F)=2$, then
$f(X,\gamma(Y,t))$ is irreducible if and only if 
$$
 g(X):=f(X+\gamma(Y,t),\gamma(Y,t))=X^2+X+\gamma(Y,t)
$$ 
is irreducible.
If $v$ denotes the normalized valuation on $F(Y)$ corresponding 
to the irreducible polynomial $Y^2+t\in F[Y]$, then $v(\gamma(Y,t))=-1$.
This implies that a zero $x$ of $g(X)$ in $F(Y)$ would satisfy $v(x)=-\frac{1}{2}$,
so $g(X)$ has no zero in $F(Y)$ and is therefore irreducible.
\end{proof}

\begin{proof}[Proof of Theorem 1]
Without loss of generality assume that $\euS\neq\emptyset$ and that 
the absolute values in $\euS$ are pairwise inequivalent. Let $F=\KS$.

The weak approximation theorem gives an element $t\in K\smallsetminus\{0,-1\}$ 
that is a uniformizer for each of the ultrametric absolute values in $\euS$.
Clearly, if $\euS$ contains an ultrametric discrete absolute value (in particular 
if ${\rm char}(K)=2$), then $t$ is not a square in $F$. Hence, by Lemma 
\ref{lem:irred}, $f(X,\gamma(Y,t))$ is irreducible over $F(Y)$.

Assume, for the purpose of contradiction, that $F$ is Hilbertian. Then there 
exists $y\in F$ such that $f(X,\gamma(y,t))$ is defined and irreducible over $F$.

Let {$|\cdot|\in\euS$}. 
If $|\cdot|$ is archimedean (this means we are in the case ${\rm char}(K)\neq 2$), 
let $\leq$ be an ordering corresponding to an extension of $|\cdot|$ to $F$, 
and let $E$ be a real closure of $(F,\leq)$. Since $\gamma(y,t)^2\geq 0$,
there exists $x\in E$ such that $f(x,\gamma(y,t))=0$ (note that the map 
$E_{\geq0}\rightarrow E_{\geq0}$, $\xi\mapsto \xi^2+\xi$ is surjective).
If {$|\cdot|$} is ultrametric and $v$ is a discrete valuation corresponding 
to an extension of $|\cdot|$ to $F$, let $E$ be a Henselization of $(F,v)$.
Since $v(\gamma(y,t))>0$ by Lemma \ref{lem:vg0}, 
$f(X,\gamma(y,t))\in\mathcal{O}_v[X]$ and 
$$
 \overline{f(X,\gamma(y,t))}=X(X+1)
$$ 
has a simple root,
so by Hensel's Lemma there exists $x\in E$ with $f(x,\gamma(y,t))=0$.

Thus in each case, $f(X,\gamma(y,t))$ has a root in $E$, so since it is of 
degree $2$ all of its roots are in $E$. Since $F$ is the intersection over 
all such $E$, all roots of $f(X,\gamma(y,t))$ lie in $F$, contradicting
the irreducibility of $f(X,\gamma(y,t))$.
\vskip0pt
\end{proof}
\eject
\centerline{{\tensc APPENDIX}: 
           {\bf THE TOTALLY ${\euS}$-ADIC IS NOT HILBERTIAN}}
\vskip5pt
\centerline{{\tensc florian pop*}}
\label{appendix}
\let\thefootnote\relax\footnotetext{* Variants of 1990/2013. 
                  Last supported by the NSF grant DMS-1101397.} 
\vskip10pt
\noindent
Let $K$ be an arbitrary field, and ${\euS}$ be a {\bf finite} set of orderings 
and/or non-trivial valuations of~$\,K$. We denote by $\,K^{\euS}\,|\,K\,$ the
maximal subextesion of a separable closure $K^{\rm sep}|K$ of $K$
in which all $v\in{\euS}$ are totally split. For $v\in{\euS}$ we will denote 
by $K_v\subset K^{\rm sep}$
a \underbar{fixed} real closure of $K$ 
with respect to $v$ in the case $v$ is an ordering, respectively a 
\underbar{fixed} Henselization of $K$ with respect to $v$ in the case 
$v$ is a valuation. Recall that $K_v\subset K^{\rm sep}$ is unique 
up to $G_K$-conjugation, where $G_K$ is the absolute Galois group 
of $K$. One has:
\vskip7pt
1) $K^{\euS}=\cap_{v\in{\euS}}\cap_{\sigma\in G_K}K_v^\sigma$.
\vskip7pt
\noindent
In particular, if $K_{v_0}=K^{\rm sep}$ for some $v_0\in{\euS}$, then 
$K^{\euS}$ does not depend on $v_0$. Therefore, without loss of 
generality, we suppose that $K_v\neq K^{\rm sep}$ for all $v\in{\euS}$.
Further, for polynomials $r(X)\in K^{\euS}[X]$ and their $G_K$-conjugates 
$r^\sigma(X)\in K^{\euS}[X]$ one has: 
\vskip7pt
2) {\it $r(X)$ has all its roots in $K^{\euS}$ iff $\,r^\sigma(X)$ has 
all its roots in $K_v$ for all $v\in{\euS}$, $\sigma\in G_K$.}
\vskip7pt
Let $L\hhb1|K$ be all the finite Galois subextensions of 
$K^{\rm sep}|K$. Then $K^{\rm sep}=\cup_L L$, and since 
$K_v\subset K^{\rm sep}$ is a strict inclusion, there exits $L\hhb1|K$ 
finite Galois such that $L$ is not contained in~$K_v$. In particular, 
since the family $(L\hhb1|K)_L$ is filtered, there exists $L\hhb1|K$ 
such that $L$ is not contained in any $K_v$, $v\in{\euS}$. Translated 
into the language of polynomials, we have the following: Let 
$p(X)\in K[X]$ be a monic polynomial having splitting field $L\hhb1|K$ 
and degree $\deg\big(p(X)\big)=[L:K]$. Then $L=k[\alpha]$ for every 
root $\alpha$ of $p(X)$. Hence the fact that $L$ is not contained 
in $K_v$ translates into:  
\vskip5pt
3) {\it There exist non-constant polynomials $p(X) \in K[X]$ 
having no roots in  $K_v$, $v\in{\euS}$.\/} 
\vskip5pt
\noindent
Equivalently, by general decomposition theory for valuations 
and orderings, it follows that $p(K_v)$ is {\it bounded away from 
zero,\/} see e.g.,~\cite{PZ}, i.e., there exists a $v$-neighborhood 
$U_v$ of $0\in K_v$ such that $U_v\cap p(K_v)$ is empty. In 
particular, for every $v\in{\euS}$ there exists $t_v\in K^\times$ 
such that $v(t_v) < v\big(p(x)\big)$ for all $x\in K_v$.$^1$\footnote{{$^1$ 
We write $v(ab)=v(a)v(b)$ for valuations, and $v(a)=\max\{a,-a\}$ 
if $v$ is an ordering.}} Taking into account that the non-zero elements 
$t\in K^\times$ approximate $0\in K_v$ simultaneously for $v\in{\euS}$, 
we get:
\vskip5pt
4) {\it Let $t\in K^\times$ satisfy $\,v(t)<v(t_v)$, $v\in{\euS}$. 
Then $v(t)<v\big(p(x)\big)$ for all $x\in K_v$, $v\in{\euS}$\/}.
\vskip5pt
We next recall the theorem on the continuity of roots in the 
following form, see e.g.,~\cite{PZ} for details:
Let $q(Y)\in K[Y]$ be a polynomial of degree $n>0$ which has 
$n$ distinct roots $y_1,\dots,y_n$ in $K$. For polynomials
$q_v(Y)\in K_v[Y]$, we define $v(q_v-q):=\max\{v(a_{vi}-a_i)\}_i$, 
where $(a_{vi})_i$ and $(a_i)_i$ are the coefficients of $q_v(Y)$,
respectively~$q(Y)$. Then for every $v\in{\euS}$ there exists 
$\delta_v\in K^\times$ such that all polynomials $q_v(Y)\in K_v[Y]$ 
of degree $n$ satisfy: If $v(q_v-q) < v(\delta_v)$, then the roots 
$y_{v1},\dots,y_{vn}$ of $q_v(Y)$ are distinct and lie in~$K_v$. 
\vskip3pt
We view polynomials $\tilde q(Y)\in K^{\euS}[Y]$ of degree~$n$ 
and their conjugates $\tilde q^{\,\sigma}(Y)\in K^{\euS}[Y]$
as polynomials in $K_v[Y]$ via $K^{\euS}\hookrightarrow K_v$,
$v\in{\euS}$. Then for $\tilde\delta\in K^\times$ one has:
\vskip5pt
5) {\it Suppose that $v(\tilde\delta)\leq v(\delta_v)$, $v\in{\euS}$, and 
$\tilde q(Y)\in K^{\euS}[Y]$ satisfy $v(\tilde q^{\,\sigma}-q) < v(\tilde\delta)$ for 
\vskip0pt \ \ \
all $v\in{\euS}$, $\sigma\in G_K$. Then all the roots of 
$\,\tilde q(Y)\in K^{\euS}[Y]$ lie in $K^{\euS}$.\/}
\vskip7pt
\noindent
{\bf Key Lemma.} \ {\it In the above notations, 
let $\delta\in K^\times$ satisfy $v(\delta)\leq v(\tilde\delta)$ if $v$ is a 
valuation, and $v(2\hhb1\delta)\leq v(\tilde\delta)$ if $v$ is an
ordering. Then  $f(X,Y):=p(X)q(Y)-t\delta\in K[X,Y]$ is absolutely 
irreducible, and for all $x\in K^{\euS}$ the specialization 
$f_x(Y):=f(x,Y)\in K^{\euS}[Y]$ splits in linear factors in~$K^{\euS}[Y]$. 
Therefore, the field $K^{\euS}$ is not Hilbertian.\/}
\begin{proof} Let $x\in K^{\euS}$ be given. Then $x^\sigma\in K_v$ 
for all $v\in{\euS}$, $\sigma\in G_K$, thus $p(x^\sigma)\in p(K_v)$. 
Hence by the definition of $t$ we have $v(t) < v\big(p(x^\sigma)\big)$,
and in particular, $p(x^\sigma)\neq0$. Further, setting 
$a:=1/p(x)$ and $u:=at$, it follows that $a^\sigma=1/p(x^\sigma)$ 
and $u^\sigma=a^\sigma t$ lie in $K^{\euS}$ and 
$v(u^\sigma)<v(1)$ for all $v\in{\euS}$, $\sigma\in G_K$. Set 
$\tilde q(Y):=a f_x(Y)=q(Y)+u\delta\in K^{\euS}[Y]$. Then
the $G$-conjugates of $\tilde q(Y)$ are
$\tilde q^{\,\sigma}(Y)=q(Y)+u^\sigma\delta\in K^{\euS}[Y]$,
thus $\tilde q^{\,\sigma}-q=(u^\sigma-u)\delta$. On 
the other hand, one has that
$v(u^\sigma-u)\leq v(u^\sigma)+v(u)<v(1)+v(1)=v(2)$
if $v$ is an absolute value, respectively 
$v(u^\sigma-u)\leq \max\{v(u^\sigma),v(u)\} < v(1)$
if $v$ is a valuation. Thus using the definition of $\delta$,  
one has that $v(\tilde q^{\,\sigma}-q)=v(u^\sigma-u)\,v(\delta) 
<v(\tilde\delta)$ for all $v\in{\euS}$, $\sigma\in G_K$. Therefore,
by~point~5) above it follows that $\tilde q(Y)$ has all its roots 
in $K^{\euS}$ and therefore, so does $f_x(Y)$.
To conclude the proof of Key Lemma, notice that $t\delta\neq0$,
and $q(Y)$ is separable. Therefore $f(X,Y)$ is absolutely irreducible, 
see e.g.,~\cite{HP} for a proof. 
\end{proof}
%
%
\noindent
{\bf Remarks}. 
\vskip2pt
1) With a proof which is virtually identical with the one 
above, one proves that the intersection of all
the {\it $v$-topological Henselizations\/} of $K$, $v\in{\euS}$, is
not Hilbertian.
\vskip2pt
2) One can ``axiomatize'' the proof of the Key Lemma and
make it work for infinite families of orderings and/or valuations,
satisfying some obvious approximation conditions.


\begin{thebibliography}{10}

\bibitem{DebesHaran}
Pierre D\`ebes and Dan Haran.
\newblock Almost {H}ilbertian fields.
\newblock {\em Acta Arithmetica}, 88:269--287, 1999.

\bibitem{Efrat}
Ido Efrat.
\newblock Absolute {G}alois groups of $p$-adically maximal {P$p$C} fields.
\newblock {\em Forum Math.}, 3:437--460, 1991.

\bibitem{FriedJarden}
Michael~D. Fried and Moshe Jarden.
\newblock {\em Field Arithmetic}.
\newblock Ergebnisse der Mathematik III {\bf 11}. Springer, 2008.
\newblock {3rd edition, revised by M. Jarden}.

\bibitem{FHV93}
Michael~D. Fried, Dan Haran, and Helmut V\"olklein.
\newblock The absolute {G}alois group of the totally real numbers.
\newblock {\em Comptes Rendus de l'Acad\'emie des Sciences Paris}, 317:95--99,
  1993.

\bibitem{FHV94}
Michael~D. Fried, Dan Haran, and Helmut V\"olklein.
\newblock Real {H}ilbertianity and the field of totally real numbers.
\newblock In N.~Childress and J.~W. Jones, editors, {\em Arithmetic Geometry},
  Contemporary Mathematics 174, pages 1--34. American Mathematical Society,
  1994.

\bibitem{HJPd}
Dan Haran, Moshe Jarden, and Florian Pop.
\newblock The absolute {G}alois group of the field of totally {$S$}-adic
  numbers.
\newblock {\em Nagoya Mathematical Journal}, 194:91--147, 2009.

\bibitem{HJPe}
Dan Haran, Moshe Jarden, and Florian Pop.
\newblock The absolute {G}alois group of subfields of the field of totally
  {$S$}-adic numbers.
\newblock {\em Functiones et Approximatio Commentarii Mathematici}, 2012.

\bibitem{HP}
B. Heinemann and A. Prestel.
\newblock Fields regularly closed with respect to finitely many valuations and orderings.
\newblock {\em Can.\ Math.\ Soc.\ Conf.\ Proc.} {\bf4} (1984), pp.297--336.


\bibitem{Pop}
Florian Pop.
\newblock Embedding problems over large fields.
\newblock {\em Annals of Mathematics}, 144(1):1--34, 1996.

\bibitem{PZ}
Alexander Prestel and Martin Ziegler.
\newblock Model theoretic methods in the theory of topological fields.
\newblock {\em Journal reine angew.\ Math.} {299/300}:318--341, 1978.

\bibitem{Taylor}
Richard Taylor.
\newblock Galois representations.
\newblock {\em Annales de la Facult\'e des Sciences de Toulouse}, 8(1):73--119,
  2004.

\end{thebibliography}
\end{document}